\documentclass[12pt,fleqn]{article}

\usepackage{amsmath,epic,amssymb}
\usepackage{amsthm,amsfonts,amscd,epsfig}
\usepackage{xcolor}

\newtheorem{theorem}{Theorem}[section]
\newtheorem{lemma}[theorem]{Lemma}
\newtheorem{remark}[theorem]{Remark}
\newtheorem{proposition}[theorem]{Proposition}
\newtheorem{corollary}[theorem]{Corollary}
\newtheorem{definition}[theorem]{Definition}
\newtheorem{example}[theorem]{Example}
\newcount\refno
\newcommand\be{\begin{equation}}
\newcommand\ee{\end{equation}}
\newcommand\bn{\begin{eqnarray}}
\newcommand\en{\end{eqnarray}}
\newcommand\bns{\begin{eqnarray*}}
\newcommand\ens{\end{eqnarray*}}
\newcommand\bd{\begin{definition}}
\newcommand\ed{\end{definition}}
\newcommand\br{\begin{remark}}
\newcommand\er{\end{remark}}
\newcommand\bt{\begin{theorem}}
\newcommand\et{\end{theorem}}
\newcommand\bp{\begin{proposition}}
\newcommand\ep{\end{proposition}}
\newcommand\bc{\begin{corollary}}
\newcommand\ec{\end{corollary}}
\newcommand\bl{\begin{lemma}}
\newcommand\el{\end{lemma}}
\newcommand\pf{\begin{proof}}

\newcommand\bN{{\mathbb N}}

\newcommand\bZ{{\mathbb Z}}

\makeindex

\def\modd#1 #2{#1\ \mbox{\rm (mod}\ #2\mbox{\rm )}}

\begin{document}

\title{On the Solutions of Three Variable Frobenius Related Problems Using Order Reduction Approach}
\author{Tian-Xiao He$^1$, Peter J.-S. Shiue$^2$, and Rama Venkat$^3$\\
{\small $^1$Department of Mathematics}\\
{\small Illinois Wesleyan University}\\
{\small Bloomington, IL 61702-2900, USA}\\
{\small $^2$ Department of Mathematical Sciences}\\
{\small University of Nevada, Las Vegas}\\
{\small Las Vegas, Nevada,  89154-4020, USA}\\
{\small $^3$ Howard R. Hughes College of Engineering}\\
{\small University of Nevada, Las Vegas}\\
{\small Las Vegas, Nevada,  89154-4020, USA}\\
}
\date{}

\maketitle

\begin{abstract}
\noindent 
This paper presents a new approach to determine the number of solutions of three variable Frobenius related problems and to find their solutions by using order reducing methods. Here, the order of a Frobenius related problem means the number of variables appearing in the problem. We present two types of order reduction methods that can be applied to the problem of finding all nonnegative solutions of three variable Frobenius related problems. The first method is used to reduce the equation of order three from a three variable Frobenius related problem to be a system of equations with two fixed variables. The second method reduces the equation of order three into three equations of order two, for which an algorithm is designed with an interesting open problem on solutions left as a conjecture. 



\end{abstract}

\setcounter{page}{1} \pagestyle{myheadings} 
\markboth{T. X. He, P. J.-S. Shiue, and R. Venkat}
{Solutions of Three Variable Frobenius Related Problems}

\section{Introduction}

Given positive integers $a_1,$ $a_2, \ldots,$ $a_\ell$ with $\gcd(a_1,a_2,\ldots, a_\ell)=1,$ we say $n\in{\bN}$ is representable if 

\[
n=m_1a_1+m_2a_2+\cdots+m_\ell a_\ell
\]
for some $m_1,m_2,\ldots, m_\ell\in{\bN}$. The well-known linear Diophantine problem asks for the largest integer $g=g(a_1,a_2,\ldots, a_\ell)$ that is not representable. The linear Diophantine problems of Frobenius have many alternative names, such as the Frobenius coin problem, the postage stamp problem, and the chicken McNugget problem (cf., for example, Bardomero and Beck \cite{BB}). 

Sylvester \cite{Syl} showed that $g(a_1,a_2)=(a_1-1)(a_2-1)$. That is for any two relatively prime positive integers $a_1$ and $a_2,$ the greatest integer that cannot be written in the form $m_1a_1 + m_2a_2$ for nonnegative integers $m_1, m_2$ is $(a_1-1)(a_2-1)-1=a_1a_2-a_1-a_2$.

A consequence of the theorem is that there are exactly $\frac{(m - 1)(n - 1)}{2}$ positive integers, which cannot be expressed in the form $am + bn$. The proof is based on the fact that in each pair of the form $(k, (m - 1)(n - 1) - k-1),$ exactly one element is expressible.

There are many stories surrounding the origin of the Chicken McNugget theorem. However, the most popular by far remains that of the Chicken McNugget. Originally, McDonald's sold its nuggets in packs of $9$ and $20$. Thus, to find the largest number of nuggets that could not have been bought with these packs creates the Chicken McNugget Theorem (the answer worked out to be $151$ nuggets). More description on the history of McNugget problem can be found, for example, in \cite{Alf,Tri}. 

In Chapman and O'Neill \cite{CON}, the McNugget number of order $3$ is defined. We call $n$ a McNugget number associated with $(p,q,\ell),$ if there exists an ordered triple $(x,y,z)$ of nonnegative integers such that 

\be\label{0.0}
p x+q y+\ell z=n
\ee
for $p,q,\ell,n >0$ and $p\leq q\leq \ell,$ where $a,b,c\geq 0$. The paper \cite{CON} considers the case of $(p, q,\ell)=(6,9,20),$ namely the nonnegative solutions of 

\be\label{0.1}
6x+9y+20 z=n.
\ee
The nonnegative solutions of \eqref{0.1} give the partitions of $n$ into parts $6,$ $9,$ and $20$. The number of those partitions are presented in the sequence {\it A214772}. It is worth noting that the coefficients of \eqref{0.1} are not pairwise relatively prime. In this paper, we consider the following related problems: (P1) For a given number $n,$ how many solutions, if exist, does \eqref{0.0} have? (P2) How to solve for all solutions with respect to given $n$? 

Alfons\'in \cite {Alf}, Hsu, Jiang, and Zhu \cite{HJZ}, and \cite{BB} present a method to solve the problems similar to  \eqref{0.1} by using the generating function. However, a large amount partial fraction steps are needed because the numbers $6,$ $9,$ and $20$ are big.  Chou et al., \cite{CS} gives a matrix method, which can be used to solve \eqref{0.1}. However, this method needs to solve a complicated system of inequalities. In this paper, we present a simple approach to solve the above two problems, which can be extended to a general three-variable Frobenius related problem. 

Denote the number of the solutions of \eqref{0.0} for a given $n$ by $N(p,q,\ell, n)$. Let $A =\{a_1,a_2,\ldots, a_k\}$ be a set of $k$ relatively prime positive integers. Let $p_A (n)$ denote the partition function of $n$ with parts in $A,$ that is, $p_A$ is the number of partitions of $n$ with parts belonging to $A$. Thus, if $a,b,$ and $c$ are relatively prime positive integers, then $N(a,b,c,n)=p_{\{a,b,c\}}(n)$. To find $p_{\{a_1,a_2,a_3\}}(n),$ where $\{a_1,a_2,a_3\}$ are relatively primes, we reduce it to $p_A(n)$ with relatively prime $A=\{a,b\}$  and use the following formula to figure $p_{\{a,b\}}(n)$: For $A=\{a,b\},$ where $(a,b)=1,$ and $n=qab+r$ with $0\leq r<ab,$ Brown, Chou, and one of the authors \cite{BCS} find  the table as follows: 

\begin{displaymath}
p_{\{a,b\}}(n)=\begin{cases}
 q+1 & \text{if $ab-a-b<r<ab;$}\\
 q & \text{if $r=ab-a-b;$}\\ 
 q+1 &\text{if $r<ab-a-b$ and $aa'(r)+bb'(r)+r=2ab;$}\\
 q & \text{if $r<ab-a-b$ and $aa'(r)+bb'(r)+r=ab,$}
 \end{cases}
\end{displaymath}
where $a'(n)$ and $b'(n)$ are defined by $a'(n)a \equiv \modd{-n} {b}$ with $1\leq a'(n)\leq b$ and $b'(n)b \equiv \modd{-n} {a}$ with $1\leq b'(n)\leq a,$respectively. Then, $N(p,q,\ell,n)$ can be found accordingly.

In the next section, we use the table to give the number of solutions of \eqref{0.0}. The process to derive the result also suggests an order reduction algorithm. In Section 3, we will present another order reduction algorithm based on the B\'ezout's lemma. A conjecture about  the solution structure is also given.

\section{The number of solutions of \eqref{0.0}}

First, we establish the following result about the number of the solutions of \eqref{0.0}.

\begin{theorem}\label{pro:1.1}
Let $N(p,q,\ell,n)$ be the number of the solutions of \eqref{0.0}, and let $(p,q)=u$. Denote the set $A=\{p/u,q/u\}$. Then

\be\label{0.3}
N(p,q,\ell,n)=\sum^{\left[ \frac{n-j\ell}{u \ell }\right]}_{k=0}p_A\left(\frac{n-j\ell}{u}-\ell k\right),
\ee
where $j\in \{ 0,1,\ldots, u-1\}$ satisfies $n-\ell j \equiv \modd{0} {u}$. 
\end{theorem}

\begin{proof}
If $(p,q)=u,$ then equation \eqref{0.0} can be written as

\be\label{0.4}
\frac{x}{u}p+\frac{y}{u}q=\frac{1}{u}(n-\ell z).
\ee
If $x,y,$ and $z$ are solutions of the above equation for a given $n,$ we need $n-\ell z \equiv \modd {0} {u}$.  If $n \equiv \modd{j_1} {u}$ with $j_1\in\{ 0,1,\ldots, u-1\},$ then we may find $z= u k+j$ for $k\geq 0$ and some $j\in \{ 0,1,\ldots, u-1\}$ such that 

\begin{align*}
n-\ell z &=n-\ell (uk+j)=(\alpha u+j_1)-\ell (uk+j)\\
&\equiv \modd{j_1-\ell j} {u}\equiv \modd{0} {u},
\end{align*}
provided that $\ell j\equiv \modd{j_1} {u}$. Since $n-\ell z=n-\ell( uk+j)\geq 0,$ we have $0\leq k\leq (n-j\ell)/u\ell$. 
Since $(p/u,q/u)=1,$ the number of the solution of \eqref{0.4} is 

\[
p_{\{p/u, q/u\}}\left( \frac{n-z\ell }{u}\right)=p_{\{ p/u,q/u\}}\left( \frac{n-j\ell}{u}-\ell k\right).
\]
and the number of solutions of \eqref{0.0} is given by \eqref{0.3}.
\end{proof}

\begin{theorem}\label{thm:1.2}
Let $N(p,q,\ell,n)$ be the number of the solutions of \eqref{0.0}, and let $(p,q)=u$. Denote $A=\{p/u,q/u\}$. Then

\be\label{0.3+1}
N(p,q,\ell,n)=\displaystyle\sum_{{0\leq z\leq [n/\ell]}\atop {u|(n-z\ell)}}p_A\left(\frac{n-z\ell}{u}\right) 
=\displaystyle\sum_{{0\leq z\leq [n/\ell]}\atop {u|(n-z\ell)}}\left(1+M(z)\right),
\ee
where 

\be\label{0.3+4}
M(z)=\frac{u}{pq}(n-z\ell-pa_1(z)-qb_1(z))
\ee
if $n-z\ell-pa_1(z)-qb_1(z)\geq 0,$ and $-1$ otherwise, i.e., $M(z)=\max\{ -1,  u(n-z\ell-pa_1(z)-qb_1(z))/(pq)\},$ $a_1(z)$ and $b_1(z)$ are the smallest nonnegative integers satisfying 

\be\label{0.3+2}
pa_1(z) \equiv \modd{(n-z\ell)} {q}\quad \mbox{and}\quad qb_1(z) \equiv \modd{(n-z\ell)} {p},
\ee
respectively. Furthermore, the set of all nonnegative solutions of the Diophantine equation \eqref{0.0} is  

\be\label{0.3+3}
\left\{ \left( \frac{q}{u}i(z)+a_1(z), (M(z)-i(z))\frac{p}{u}+b_1(z), z\right): 0\leq i(z)\leq M(z)\right\}
\ee 
for all $0\leq z\leq [n/\ell]$ with $u|(n-z\ell),$ where $M(z)$ is defined by \eqref{0.3+4} if $\frac{u}{pq}(n-z\ell-pa_1(z)-qb_1(z)\geq 0,$ otherwise, 
the nonnegative solution set \eqref{0.3+3} does not exist.
\end{theorem}

\begin{proof}
We may use the formula shown in Binner \cite[Corollary 17]{Bin} to find the number $p_{\{ p/u,q/u\}}\left(\frac{n-z\ell}{u}\right)$. Denote $a=p/u,$ $b=q/u,$ and $m(z)=(n-z\ell)/u$ $(0\leq z\leq \lfloor (n/\ell)\rfloor)$. 
Thus, $\gcd (a,b)=1,$ and we need $m(z)$ to be a nonnegative integer. In general, from \cite{Bin} the number of nonnegative solutions of $ax+by=m(z)$ is 
\be\label{0.4+1}
p_{a,b}(m(z))=1+\frac{m(z)-aa_1(z)-bb_1(z)}{ab}, 
\ee
if the second term on the right-hand side of \eqref{0.4+1} is nonnegative, where $a_1(z)$ is the remainder when $m(z)a^{-1}$ is divided by $b,$ $b_1$ is the remainder when $m(z)b^{-1}$ is divided by $a,$ and $a^{-1}$ and $b^{-1}$ are the modular inverse of $a$ with respect to $b$ and $b$ with respect to $a,$ respectively, namely,

\[
m(z)a^{-1} \equiv \modd{a_1(z)} {b}\quad \mbox{and} \quad m(z)b^{-1} \equiv \modd{b_1(z)} {a}.
\]
It is obvious that the last two equations are equivalent to 

\[
m(z) \equiv \modd{aa_1(z)} {b}\quad \mbox{and} \quad m(z) \equiv \modd{bb_1(z)} {a}.
\]
Thus, $a_1(z)\geq 0$ and $b_1(z)\geq 0$ can be found by solving (cf. \ Remark $16$ of \cite{Bin}) 
\be\label{0.4+2}
ax \equiv \modd{m(z)} {b}\quad \mbox{and}\quad by \equiv \modd{m(z)} {a}
\ee
for $x \equiv \modd{a_1(z)} {b}$ and $y \equiv \modd{b_1(z)} {a}$. 
From \eqref{0.4+2} we have \eqref{0.3+2} by noticing $a=p/u$ and $b=q/u$. Denote $M(z)=(m(z)-aa_1(z)-bb_1(z))/ab$ if it is nonnegative. Substituting $m(z)=(n-z\ell)/u,$ $a=p/u$ and $b=q/u$ into $M(z),$ we obtain \eqref{0.3+4}. Then, from \eqref{0.4+1} we have 

\[
p_{p/u,q/u}(m(z))=1+M(z)=1+\frac{u}{pq}(n-z\ell-pa_1(z)-qb_1(z))
\]
if $\frac{u}{pq}(n-z\ell-pa_1(z)-qb_1(z)\geq 0,$ which implies \eqref{0.3+1}.

Furthermore, from \cite{Bin} the set of all nonnegative solutions of $ax+by=m(z)$ is 

\be\label{0.4+3}
\{ (bi+a_1(z), (M(z)-i)a+b_1(z)): 0\leq i\leq M(z)\}
\ee
when $M(z)=\frac{u}{pq}(n-z\ell-pa_1(z)-qb_1(z))\geq 0,$ which implies \eqref{0.3+3}. If $\frac{u}{pq}(n-z\ell-pa_1(z)-qb_1(z)< 0,$ then the set shown in \eqref{0.4+3} does not exists, which implies $1+M(z)$ in \eqref{0.3+1} must be zero. Hence, $M(z)=-1$ in this case, which completes the proof of the theorem.
\end{proof}

We use the formula of the number of nonnegative solutions of $ax+by=m$ given in \cite{Bin} with modification for the case the nonnegative solution does not exist. Binner's formula is equivalent to the one given in Tripathi \cite{Tri00}. In addition, the paper \cite{Rev} gives the pretty much the same idea as that used in \cite{Bin}.

We may use Theorems \ref{pro:1.1} and \ref{thm:1.2} to find $N(6,9,20,n)$. To use formula 
\eqref{0.3}, we rewrite \eqref{0.1} as  

\be\label{0.5}
6x+9y=n-20z
\ee
and 

\be\label{0.9}
2x+3y=\frac{1}{3}(n-20z).
\ee
Since $3$ is a factor of the left-hand side of \eqref{0.5}, we have 

\[
n-20z \equiv \modd{0} {3},
\]
which implies 

\[
n \equiv \,20z \equiv \modd{-z} {3}.
\]
Thus, we need to consider three cases for $n \equiv \modd{0,1,2} {3},$ respectively, namely 

\begin{align}
z &\equiv \modd{0} {3},\label{0.6}\\
z &\equiv \modd{1} {3},\quad \mbox{and}\label{0.7}\\
z &\equiv \modd{2} {3},\label{0.8}
\end{align}
respectively. 

In the case of \eqref{0.6}, by noticing $n-20z=6a+9b\geq 0$ and $z=3k$ for $k\geq 0,$ we have 

\[
0\leq 3k\leq \left[ \frac{n}{20}\right].
\]
Thus, for $n \equiv \modd{0} {3},$ from Theorem \ref{pro:1.1} the number of the solutions of \eqref{0.5} is 

\[
N(6,9,20, n)=\sum^{\left[ \frac{n}{60 }\right]}_{k=0}p_{\{ 2,3\}}\left(\frac{n}{3}-20 k\right), \quad n \equiv \modd{0} {3}. 
\]

In the case of \eqref{0.7}, because $n \equiv \modd{1} {3},$ $z \equiv \modd{2} {3}$. By setting $z=3k+2,$ we have the number of the solutions of 

\[
2x+3y=\frac{1}{3}(n-20z)=\frac{1}{3}(n-20 (3k+2))
\]
is 

\[
N(6,9,20, n)=\sum^{\left[ \frac{n-40}{60 }\right]}_{k=0}p_{\{ 2,3\}}\left(\frac{n-40}{3}-20 k\right),\quad n \equiv 
\modd{1} {3}.  
\]

Finally, in the case of \eqref{0.8}, because $n \equiv \modd{2} {3},$ $z \equiv \modd{1} {3}$. By setting $z=3k+1,$ we obtain the number of the solutions of

\[
2a+3b=\frac{1}{3}(n-20z)=\frac{1}{3}(n-20 (3k+1))
\]
is 

\[
N(6,9,20, n)=\sum^{\left[ \frac{n-20}{60 }\right]}_{k=0}p_{\{ 2,3\}}\left(\frac{n-20}{3}-20 k\right), \quad n \equiv 
\modd{2} {3}. 
\]

We now use formula \eqref{0.3+1} to count the solutions of problem \eqref{0.1}. Since $u=\gcd(p,q)=\gcd (6,9)=3,$ to have $u|(n-20z)$ for $n \equiv \modd{0,1,2} {3}$ or $n=3r,3r+1,3r+2$ $(r\in{\bN}),$ we need 
$z \equiv \modd{0,2,1} {3}$ or $ z=3k, 3k+2, 3k+1$ $(k\in {\bN}\cup\{0\}),$ respectively. Thus,

\be\label{Ex:1}
M(z)=\frac{3}{54}(n-20z-6a_1(z)-9b_1(z))=\frac{1}{18}(n-20z-6a_1(z)-9b_1(z)).
\ee
For $n=3r$ and $z=3k,$ we have $a_1(3k)$ and $b_1(3k)$ satisfying 

\be\label{Ex:2}
6a_1(3k) \equiv \modd{(3r-60k)} {3}\quad \mbox{and}\quad 9b_1(3k) \equiv \modd{(3r-60k)} {2}.
\ee
For $n=3r+1$ and $z=3k+2,$ we have $a_1(3k+2)$ and $b_1(3k+2)$ satisfying 

\begin{align}\label{Ex:3}
6a_1(3k+2)&\equiv (3r-60k-39) \equiv \modd{0} {3}\quad \mbox{and}\nonumber\\ 
9b_1(3k+2)&\equiv \modd{(3r-60k-39)} {2}.
\end{align}
For $n=3r+2$ and $z=3k+1,$ we have $a_1(3k+1)$ and $b_1(3k+1)$ satisfying 

\begin{align}\label{Ex:4}
6a_1(3k+1)&\equiv (3r-60k-18) \equiv \modd{0} {3}\quad \mbox{and}\nonumber\\
9b_1(3k+1)&\equiv \modd {(3r-60k-18)} {2}.
\end{align}
Hence, the number of solutions of \eqref{0.1} is 

\begin{align}\label{Ex:5}
N(6,9,20,n)&=\displaystyle\sum_{{0\leq z\leq [n/\ell]}\atop {3|(n-20z)}}p_{\{ 2,3\}}\left(\frac{n-20z}{3}\right) =\displaystyle\sum_{{0\leq z\leq [n/\ell]}\atop {3|(n-20z)}}(1+M(z))\nonumber\\
&=\displaystyle\sum_{{0\leq z\leq [n/\ell]}\atop {3|(n-20z)}}\left(1+\frac{1}{18}(n-20z-6a_1(z)-9b_1(z)\right),\nonumber\\
\end{align}
where $a_1(z)$ and $b_1(z)$ satisfy \eqref{Ex:2}-\eqref{Ex:4}.

The nonnegative solutions of the problem \eqref{0.1} are 

\be\label{Ex:6}
\left\{ \left( 3i(z)+a_1(z), 2(M(z)-i(z))+b_1(z), z\right): 0\leq i(z)\leq M(z)\right\}
\ee 
for all $0\leq z\leq [n/20]$ with $3|(n-20z),$ where $M(z)$ is given in \eqref{Ex:1}. 
More specifically, for $n=3r,$ $3r+1,$ and $3r+2,$ we have $z=3k,$ $3k+2,$ and $3k+1,$ respectively, 
and the corresponding 

\be\label{Ex:7}
M(3k)=\frac{1}{18}(3r-60k-6a_1(3k)-9b_1(3k))
\ee
for $0\leq k\leq [r/20],$ 

\be\label{Ex:8}
M(3k+2)=\frac{1}{18}(3r-60k-39-6a_1(3k+2)-9b_1(3k+2))
\ee
for $0\leq k\leq [(r-13)/20],$ and 

\be\label{Ex:8}
M(3k+1)=\frac{1}{18}(3r-60k-39-6a_1(3k+1)-9b_1(3k+1))
\ee
for $0\leq k\leq [(r-6)/20].$

For instance, if $n=39,$ then $M(3k)$ is given by \eqref{Ex:7}, where $a_1(3k)$ and $b_1(3k)$ satisfying  

\[
2a_1(3k) \equiv \modd{(13-20k)} {3}\quad \mbox{and}\quad 3b_1(3k) \equiv \modd{(13-20k)} {2}
\]
for $k=0$. Thus, $a_1(3k)=a_1(0)=2$ and $b_1(3k)=b_1(0)=1,$ and 

\[
M(3k)=M(0)=\frac{1}{18}(39-6\cdot 2-9\cdot 1)=1.
\]
Consequently, the problem $6x+9y+20z=39$ has $1+M(0)=2$ solutions, which are 

\[
\displaystyle{\cup_{0\leq i\leq M(0)\atop 0\leq k\leq [13/20]}}\{(3i+2, 2(M(3k)-i)+1,3k)\}
=\{(2,3,0), (5,1,0)\}.
\]

If $n=46 \equiv  \modd{1} {3},$ then $M(3k+2)$ is given by \eqref{Ex:8}, where $a_1(3k+2)$ and $b_1(3k+2)$ satisfying 

\[
2a_1(3k+2) \equiv \modd{(2-20k)} {3}\quad \mbox{and}\quad 3b_1(3k+2) \equiv \modd{(2-20k)} {2}
\]
for $k=0$. Thus, $a_1(3k+2)=a_1(2)=1$ and $b_1(3k+2)=b_1(2)=0,$ and 

\[
M(3k+2)=M(2)=\frac{1}{18}(6-6\cdot 1-9\cdot 0)=0.
\]
Consequently, the problem $6x+9y+20z=46$ has $1+M(2)=1$ solution, which is 

\[
\displaystyle{\cup_{0\leq i\leq M(2)\atop 0\leq k\leq [2/20]}}\{(3i+1, 2(M(2)-i)+0,3k+2)\}
=\{(1,0,2)\}.
\]

If $n=50 \equiv \modd{2} {3},$ then $M(3k+1)$ is given by \eqref{Ex:8}, where $a_1(3k+1)$ and $b_1(3k+1)$ satisfying 

\[
2a_1(3k+1) \equiv \modd{(10-20k)} {3}\quad \mbox{and}\quad 3b_1(3k+1) \equiv \modd{(10-20k)} {2}
\]
for $k=0$. Thus, $a_1(3k+1)=a_1(1)=2$ and $b_1(3k+1)=b_1(1)=0,$ and 

\[
M(3k+1)=M(1)=\frac{1}{18}(30-6\cdot 2-9\cdot 0)=1.
\]
Consequently, the problem $6x+9y+20z=50$ has $1+M(1)=2$ solutions, which are 

\[
\displaystyle{\cup_{0\leq i\leq M(1)\atop 0\leq k\leq [10/20]}}\{(3i+2, 2(M(1)-i)+0,3k+1)\}
=\{(2,2,1), (5,0,1)\}.
\]

To find all solutions of \eqref{0.0}, we may use Theorem \ref{thm:1.2} to determine the number of solutions and use the following proposition to find each solution of \eqref{0.0}.

\begin{proposition}\label{pro:1.2}
Let $(x,y)=(x_0,y_0)$ be a solution of $p x+q y=m$. Then all solutions of the equations are $(x_0+q t, y_0-p t)$ for all $t\in {\bZ}$. 
\end{proposition}

\begin{example} \label{ex:2.1} 
Consider the equation 

\[
6x+9y+20z=84,
\]
and rewrite it as 

\[
2x+3y=\frac{1}{3}(84-20z).
\]
Since $84 \equiv \modd{0} {3},$ we have the number of the solutions of the equation 

\begin{align*}
N(6,9,20,84)&=\sum^{[84/60]}_{k=0}p_{\{2,3\}}\left( \frac{84}{3}-20k\right)\\
&=p_{\{2,3\}}(28)+p_{\{2,3\}}(28-20).
\end{align*}
Form the formula for $p_{\{a,b\}}(n)$ presented in the Introduction by noting $a=2,$ $b=3,$ and $n=28,$ we have $28=4\cdot 6+4$ with $q=4$ and $r=4$ satisfying $2\cdot 3-2-3=1<r<6=2\cdot 3$. Thus $p_{\{2,3\}}(28)=q+1=5$. While for $a=2,$ $b=3,$ and $n=8,$ we have $8=1\cdot 6+2$ with $q=1$ and $r=2$ satisfying $1<r<6$. Thus $p_{\{2,3\}}(8)=q+1=2$ and 

\[
N(6,9,20,84)=p_{\{2,3\}}(28)+p_{\{2,3\}}(28-20)=7.
\]

The nonnegative solutions of $2x+3y=28-20z=28$ can be found from Proposition \ref{pro:1.2} and Euclidean algorithm as follows:
 
\begin{align*}
x&=2,\quad y=8,\quad z=0;\\
x&=5,\quad y=6,\quad z=0;\\
x&=8,\quad y=4,\quad z=0;\\
x&=11,\quad y=2,\quad z=0;\\
x&=14,\quad y=0,\quad z=0.
\end{align*}
The non-negative solutions of $2x+3y=28-20z=8$ are 

\begin{align*}
x&=1,\quad y=2,\quad z=1;\\
x&=4,\quad y=0,\quad z=1.
\end{align*}
\end{example}

For problems for which the Euclidean algorithm is inefficient, an alternate approach is presented below. Consider the equation to be $px+qy=m$ with the given triple $(p,q,m)$, for which $m$ is divisible by $d=\gcd (p,q)$, we are seeking integer solutions (x,y).
\begin{enumerate}
\item[Step 1:] Want a nonnegative integer solution $(x_0,y_0)$ for
\[
px' + qy' =m
\]
\item[Step 2:] Use modular arithmetic on the equation in Step $1$. 
\[
qy' \equiv \modd{m} {p}
\]
\item[Step 3:] Solve the above congruence  using the smallest possible nonnegative value $y'$. Substituting the value of $y'$ in the equation finds $x'$. If $x'$ and $y'>0,$ then, $x_0 = x'$ and $y_0 = y'$. Then step 4 can be skipped.
\item[Step 4:] If  $x' < 0,$ then use Bezout's Lemma to obtain the first nonnegative solution$(x_0 , y_0 )$ from  $(x',y')$ with a positive integer $n$ such that $x_{0}=x'+\frac{nq}{d}> 0$ and $y_{0}=y'- \frac{np}{d} > 0,$ where $d=\gcd (p,q)$. If no such $n$ exists, then there are no nonnegative solutions to the problem. For example, $3x+4y=5$ does not have any nonnegative solutions.
\end{enumerate}
Illustration of the algorithm is as follows.
Consider
\[
3x+5z=14 \quad
\]
Note that the common divisor, d,  of 3 and 5 is 1.
\begin{enumerate}
\item[Step 1:]Want nonnegative integer solution, $(x_0, y_0),$ for $3x_0+5y_0 =14$
\item[Step 2:]$3x' + 5y' \equiv \modd{14} {3}$ implies $2y' \equiv \modd{2} {3}$ and $y' \equiv \modd{1} {3}$. Hence, $y' = 1$ and  $x' = 3$.
\item[Step 3:] Since $x'= 3 > 0,$ and $y'= 1 >0,$ we have 
$(x_0,y_0)=(x',y')=(3,1)$.
\end{enumerate}

\begin{remark}\label{rem:2.1} 
We may also use the method presented in Binner \cite{Bin} to solve $2x+3y=28$ for obtaining the set of all nonnegative solutions $(x,y)$. We set $2x \equiv \modd{28} {3}$ and $3y \equiv 28 \equiv \modd{0} {2}$ to obtain $x \equiv \modd{2} {3}$ and $y \equiv \modd{0} {2},$ respectively. Thus, $a_1=2,$ $b_1=0,$ and $M=(28-2\cdot 2-3\cdot 0)/6=4$. Consequently, the nonnegative solutions of $2x+3y=28$ are $\{(3i+2, 2(4-i)+0):0\leq i\leq 4\},$ which are exactly the same as what we have obtained. 
\end{remark}

\begin{remark}\label{rem:2.2} 
From the theory of partition (cf. \ \cite{And}), the generating function of the sequence $(N_n=N(p,q,\ell,n))_{n\geq 0}$ is 

\be\label{0.3-2}
\sum_{n\geq 0}N_nt^n=\frac{1}{(1-t^p)(1-t^q)(1-t^\ell)}.
\ee
From Taylor's expansion, we have 

\be\label{0.4-2}
N_n=\left. \frac{1}{n!}\frac{d^n}{dt^n}\frac{1}{(1-t^p)(1-t^q)(1-t^\ell)}\right|_{t=0}.
\ee
For smaller coefficients $p,q,$ and $\ell,$ we may find $N_n$ more efficiently. For instance, 
if $p=1,$ $q=2,$ and $\ell=3,$ then the solution number $N_n$ can be found by using the partial fraction technique shown below. Let $\omega=e^{2\pi i/3}=\cos(2\pi/3)+i\sin(2\pi/3)$. Then 

\begin{align*}
\frac{1}{(1-t)(1-t^2)(1-t^3)}&=\frac{1}{(1-t)^3(1+t)(1-\omega t)(1-\omega^2t)}\\
&=\frac{1}{6(1-t)^3}+\frac{1}{4(1-t)^2}+\frac{17}{72(1-t)}+\frac{1}{8(1+t)}\\
&+\frac{1}{9(1-\omega t)}+\frac{1}{9(1-\omega^2t)}\\
&=\sum_{n\geq 0}\left( \frac{(n+3)^2}{12}-\frac{7}{72}+\frac{(-1)^n}{8}+\frac{2}{9}\cos\frac{2n\pi}{3}\right)t^n\\
&=\sum_{n\geq 0}N_n t^n
\end{align*}
for $|t|<1$. Since 

\[
\left| -\frac{7}{72}+\frac{(-1)^n}{8}+\frac{2}{9}\cos\frac{2n\pi}{3}\right|\leq \frac{32}{72}<\frac{1}{2},
\]
and $N_n$ must be an integer, we obtain $N_n=\left< \frac{(n+3)^2}{12}\right>,$ where $<\alpha>$ $(\alpha \not= 1/2)$ is referred to as the closest integer to $\alpha$. For instance, 
the solution number for $x+2y+3z=14$ is 

\[
N_{14}=\left< \frac{(14+3)^2}{12}\right>=24. 
\] 
\end{remark}

\section{Order reduction algorithm and exhaustive method for solving \eqref{0.0}}

To find all solutions of \eqref{0.0} we may use the following B\'ezout's Lemma (cf. \ Millman, Kahn, and one of the authors \cite[Theorem 9]{MSK}).

\begin{lemma}\label{lem:1.2}
The linear Diophantine equation $px+qy=r$ has a solution if and only if $r$ is divisible by $d=(p,q)$. Furthermore, if $(x_0,y_0)$ is any particular solution of this equation, then all other solutions are given by 

\be\label{0.10}
x'=x_0+\frac{q}{d}k\quad y'=y_0-\frac{p}{d}k,
\ee
where $k$ is an arbitrary integer.
\end{lemma}

Let $p$ and $q$ be integers with the greatest common divisor $d$. Then from B\'ezout's lemma there exist integers $x$ and $y$ such that $px+qy=d$. More generally, the integers of the form $px+qy$ are the multiples of $d$. Expressions $x'$ and $y'$ shown in \eqref{0.10} are clearly true. 

We are going to use B\'ezout's lemma to solve a problem with two different features: (1) We are solving the Diophantine equation \eqref{0.0} of order $3,$ and (2) we are seeking all non-negative  solutions. 

Our algorithm is based on an order reducing technique. More precisely, let $p,q,\ell,$ and $n\in{\bN}$ with $p\leq q\leq \ell,$ and let $x,y,z\in{\bZ}$. A linear Diophantine equation $ax+by+cz=d$ of three-variables is reduced to the following three Diophantine equations of two variables after setting $x,y,z=0,$ respectively.    Suppose one of $(p,q),$ $(q,\ell),$ and $(\ell,p)$ are divisors of $n$. B\'ezout's lemma \ref{lem:1.2} 
shows that at least one of the Diophantine equations 

\begin{align}\label{0.11}
px+qy&=n\nonumber\\
qy+l z&=n\quad \mbox{and}\nonumber\\
px+lz&=n
\end{align}
has solutions. For instance, if $(p,q)|n,$ then one pair of non-negative solutions $(x_0,y_0)$ of the first equation $px+qy=n$ can be found easily by using an extended Euclidean algorithm easily. Particularly, if $(p,q)=p$ (or $q$), then it is easier to obtain a pair of solutions as $(n/p,0)$ (or $(0,n/q)$), and all pairs of solutions of the equation can be represented in the form 

\be\label{0.12}
\left( x_0+k\frac{q}{(p,q)},\, y_0-k\frac{p}{(p,q)}\right),
\ee
where $k$ is an arbitrary integer. The set of all those solutions $(x,y,0)$ is denoted by $S,$ i.e., 

\[
S=\left\{\left( x_0+k\frac{q}{(p,q)},\, y_0-k\frac{p}{(p,q)}\right): k\in {\bZ}\right\}.
\]
We are seeking nonnegative solutions of $px+qy=n,$ i.e., a particular solution pair $(x_0, y_0),$ where $x_0, y_0\geq 0,$ and all solution pairs in \eqref{0.12} for $k$ such that 

\be\label{0.13}
k\in K=\{ k\in {\bZ}: -\frac{x_0(p,q)}{q}\leq k_1\leq \frac{y_0(p,q)}{p}\},
\ee
that is $x_0+k\frac{q}{(p,q)},$ $y_0-k\frac{p}{(p,q)}\geq 0$.

In the above algorithm, we must assume one of the conditions, $(p,q)|n,$ $(q,\ell)|n,$ and $(p,\ell)|n,$ holds; otherwise, our algorithm fails because if $gcd(p,q)$ does not divide $n$, then clearly $px + qy = n$ has no solutions since $gcd(p, q)$ divides the left hand side but not the right hand side.

Consequently, if $p,q,$ and $l$ are pairwise coprime numbers, then $px+qy+lz=n$ can be solved by using our algorithm. As what we have defined before, the set of all solutions of the first equation in \eqref{0.11} is denoted by $S_1$. Similarly, we let $S_2$ and $S_3$ denote the sets of the solutions of the second and the third equations of \eqref{0.11}, respectively. 

It is obvious that for any integers $a,b,$ and $c,$ with $a+b+c\not= 0$ and row vectors $s_i\in S_i,$ $i=1,2,$ and $3,$ $(as_1+bs_2+cs_3)/(a+b+c)$ is also a solution of $px+qy+\ell z=n$ because 

\[
\frac{as_1+bs_2+cs_3}{a+b+c}\cdot (p,q,\ell)^T=\frac{1}{a+b+c}(an+bn+cn)=n.
\]
Because a solution obtained by using our order reducing method has at least one zero component,  we can see how important it is to use these linear combinations complete the set of all solutions of \eqref{0.0}, say to calculate the solutions with non-zero components. We will also demonstrate how to build those linear combinations by using some examples. 

\begin{proposition}\label{pro:3.1}
Let sets $S_i,$ $i=1,2,3,$ be defined before, and let $a,b,$ and $c$ be any non integers, with $a+b+c\not= 0$. We have notation 

\begin{align*}
\widehat S_1&=\left\{ \left( 0, \, y_{2,0}+k_2\frac{l}{(q,l)},\, z_{2,0}-k_2\frac{q}{(q,l)}\right): y_{2,0}, z_{2,0}\geq 0, k_2\in K_1\right\},\\
\widehat S_2&=\left\{ \left( x_{3,0}+k_3\frac{l}{(p,l)},\, 0, \, z_{3,0}-k_3\frac{p}{(p,l)}\right): x_{3,0}, z_{3,0}\geq 0, k_3\in K_2\right\}, \\
\widehat S_3&=\left\{ \left( x_{1,0}+k_1\frac{q}{(p,q)},\, y_{1,0}-k_1\frac{p}{(p,q)}, \, 0\right): x_{1,0},y_{1,0}\geq 0, k_1\in K_3\right\},
\end{align*}
where $(y_{2,0}, z_{2,0})$, $(x_{3,0}, z_{3,0})$ and $(x_{1,0}, y_{1,0})$ are solutions of $qy+\ell z=n$, $px+\ell z=n$, and $px+qy=n$, respectively, and 

\begin{align*}
K_1&=\{ k\in {\bZ}: -\frac{y_{2,0}(q,l)}{l}\leq k \leq \frac{z_{2,0}(q,l)}{q}\},\\
K_2&=\{ k\in {\bZ}: -\frac{x_{3,0}(p,l)}{l}\leq k \leq \frac{z_{3,0}(p,l)}{p}\},\\
K_3&=\{ k\in {\bZ}: -\frac{x_{1,0}(p,q)}{q}\leq k \leq \frac{y_{1,0}(p,q)}{p}\}.
\end{align*}
Then all non-negative elements $(x,y,z)$ in the set 

\[
\left\{ \frac{as_1+bs_2+cs_3}{a+b+c}: a+b+c\not= 0, s_i\in \widehat S_i, i=1,2,3\right\}
\]
is a solution of \eqref{0.0}.
\end{proposition}

The proof is obvious from the above discussion, and is, therefore, omitted. 

To avoid using linear combinations, we may use the following exhaustive method. We assume 
$p\leq q\leq \ell$ and let  

\[
S(z_i)=\{ (x,y, i): px+qy=n-z_i,\,\,z_i=i\ell \,\, with\,\, 0\leq z_i\leq n\}.
\]
Then the solution set of \eqref{0.0} is $\cup^{[n/l]}_{i=0} S(z_i)$. 

\begin{example}\label{ex:3.1}
Consider the equation 

\be\label{3.1}
6x+9y+20z=84
\ee
and reduce it to three Diophantine equations of order $2$:

\begin{align}
&9y+20z=84,\label{3.3}\\
&6x+20z=84\label{3.4}\\
&6x+9y=84,\label{3.2}
\end{align}
for the cases of $z=0,$ $x=0,$ and $y=0$ in \eqref{3.1}.

Equation \eqref{3.2} can be written as $2x+3y=28$. Hence, $y=0$ yields a particular solution of \eqref{3.2} as $x=14$ and $y=0$. The solution set $\widehat S_3$ is 

\begin{align*}
\widehat S_3&=\left\{ \left( 14+k_1\frac{9}{(6,9)},\, 0-k_1\frac{6}{(6,9)}, \,0\right): -\frac{14}{3}\leq k_1\leq 0\right\}\\
&=\left\{ (14,0,0), (11,2,0), (8,4,0), (5,6,0), (2,8,0)\right\}.
\end{align*}

Similarly, from equation \eqref{3.3} we have 

\[
y=\frac{84-20z}{9},
\]
which shows \eqref{3.3} has no non-negative integer solutions $(y,z)$. Hence, 
the solution set of equation \eqref{3.3} is 

\[
\widehat S_1=\phi.
\]

Finally, from equation \eqref{3.4} we have 

\[
x=\frac{42-10z}{3}=12-3z+\frac{6-z}{3}.
\]
To have non-negative solutions of \eqref{3.4}, we must have $z=0$ or $z=3,$ which implies 

\[
\widehat S_2=\left\{ (14,0,0), (4,0,3)\right\}.
\]
Hence, we have found $6$ distinct solutions of \eqref{3.1}. The $7$th solution is from a linear combination of the elements $(8,4,0),$ $(11,2,0)$ and $(4,0,3)$ of the solution set $\cup_{i=1}^3\widehat S_i$ with the coefficients $a=1,$ $b=-1$ and $c=1$:

\begin{align*}
&\frac{a}{a+b+c}(8,4,0)+\frac{b}{a+b+c}(11,2,0)+\frac{c}{a+b+c}(4,0,3)\\
&=(8,4,0)-(11,2,0)+(4,0,3)=(1,2,3).
\end{align*}
Hence, the solutions of $6x+9y+20\ell=84$ are  

\be\label{S}
S=\left\{ (14,0,0), (11,2,0), (8,4,0), (5,6,0), (2,8,0),(4,0,3), (1,2,3)\right\}.
\ee

If we use the exhaustive method to check our results, then we obtain

\begin{align*}
S(0)&=\left\{ (14,0,0), (11,2,0), (8,4,0), (5,6,0), (2,8,0)\right\},\\
S(20)&=\left\{ (x,y,1):6x+9y=84-20=64\right\}=\phi,\\
S(40)&=\left\{ (x,y,2):6x+9y=84-40=44\right\}=\phi,\\
S(60)&=\left\{ (x,y,3):6x+9y=84-60=24\right\}=\left\{ (4,0,3), (1,2,3)\right\}.
\end{align*}
Hence, the solutions of $6x+9y+20\ell=84$ are as the same as \eqref{S}
\end{example}

However, sometimes we have many solutions of $px+qy+rz=n$ that are not from the union $\cup^3_{i=1}\widehat S_i$. Therefore, we need calculate the number of the solutions and develop an efficient way (see below) to find the solutions in the form of $(as_1+bs_2+cs_3)/(a+b+c),$ where $a,b,c\in {\bZ}$ and $s_i\in \widehat S_i,$ $i=1,2,3$. For instance, we may consider $(p,q,l)=(1,2,3),$ and the corresponding solution sets of $x+2y+3z=14$ by using our algorithm are $S_1=\{(0, 7+3k_3, -2k_3):k_3\in {\bZ}\},$ $S_2=\{ (14+3k_2,0,-k_2):k_2\in {\bZ}\},$ and $S_3=\{(2k_1, 7-k_1,0):k_1\in {\bZ}\},$ respectively.  Hence, 

\begin{align}
& \widehat S_1=\{(0,7,0), (0,4,2), (0,1,4)\},\label{S3}\\
& \widehat S_2=\{(14,0,0), (11,0,1), (8,0,2), (5,0,3), (2,0,4)\},\label{S2}\\
& \widehat S_3=\{(0,7,0), (2,6,0), (4,5,0), (6,4,0), (8,3,0), (10,2,0),\nonumber\\
&\qquad  (12,1,0), (14,0,0)\}.\label{S1}
\end{align}
However, by using Theorem \ref{pro:1.1}, we may find that the number of the solutions of $x+2y+3z=14$ is 

\begin{align*}
\sum^{[14/3]}_{k=0}p_{1,2}(14-3k)&=p_{1,2}(14)+p_{1,2}(11)+p_{1,2}(8)+p_{1,2}(5)+p_{1,2}(2)\\
&= 8+6+5+3+2=24,
\end{align*}
where $p_{1,2}(n),$ $n=14,11,8,5,$ and $2,$ are found by using Theorem \ref{pro:1.1}.

\begin{remark}\label{rem:3.1} 
For small $p$ and $q,$ the number of the solutions, denoted by $p_{p,q}(n),$ of the Diophantine equation $px+qy=n$ can be also found by using the theory of partition, which is similar to the process to derive \eqref{0.3} and \eqref{0.4}. More precisely, from the theory of partition (cf. \ \cite{And}), the generating function of the sequence $(P_n=p_{p,q}(n))_{n\geq 0}$ is 

\be\label{0.3-3}
\sum_{n\geq 0}P_nt^n=\frac{1}{(1-t^p)(1-t^q)}.
\ee
From Taylor's expansion, we have 

\be\label{0.4-3}
P_n=\left. \frac{1}{n!}\frac{d^n}{dt^n}\frac{1}{(1-t^p)(1-t^q)}\right|_{t=0}.
\ee

Let $N_n$ and $P_n$ be defined by \eqref{0.3-2}-\eqref{0.4-2} and \eqref{0.3-3}-\eqref{0.4-3}, respectively. Then 

\[
N_n=\sum^{[n/\ell]}_{k=0}P_{n-\ell k}. 
\]
Consequently, by using a straightforward exhaustive method for solving the Diophantine equation \eqref{0.0}, we have

\[
\frac{d^n}{d t^n}\frac{1}{(1-t^p)(1-t^q)(1-t^\ell)}=n!\sum^{[n/\ell]}_{k=0}\frac{1}{(n-\ell k)!}
\frac{d^{n-\ell k}}{dt^{n-\ell k}}\frac{1}{(1-t^p)(1-t^q)}.
\]

For smaller coefficients $p,q,$ and $\ell,$ we may find $P_n$ more efficiently. For instance, 
if $p=1$ and $q=2,$ then the solution number $P_n$ can be found by using the partial fraction technique shown below. Since 

\[
\frac{1}{(1-t)(1-t^2)}=\frac{1}{4}\left(\frac{1}{1+t}+\frac{1}{1-t}+\frac{2}{(1-t)^2}\right)\quad (|t|<1),
\]
the expansion of the power series can be written as 

\[
\sum_{n\geq 0}P_nt^n=\sum_{n\geq 0}\frac{(-1)^n+1+2(n+1)}{4}t^n.
\]
Consequently, 

\[
P_n=\frac{1}{4}(2n+3+(-1)^n),
\]
which gives $p_{1,2}(14)=P_{14}=8,$ $p_{1,2}(11)=P_{11}=6,$ $p_{1,2}(8)=P_8=5,$ $p_{1,2}(5)=P_5=3$, and $p_{1,2}(2)=P_2=2$. 
\end{remark}

Since 

\begin{align}
\cup^3_{i=1} \widehat S_i =&\{ (0,7,0), (2,6,0), (4,5,0), (6,4,0), (8,3,0), (10,2,0),\nonumber \\
&(12,1,0), (14,0,0), (11,0,1), (8,0,2), (5,0,3), (2,0,4), (0,4,2), (0.1.4)\}\nonumber\\
&
\label{S4}
\end{align}
we need to find $10$ more solutions by using the linear combination $(as_1+bs_2+cs_3)/(a+b+c),$ where $s_i\in S_i,$ $i=1,2,3$ and $a+b+c\not= 0$. 

The algorithm to establish those linear combinations is to find sets $\{(a,b,c):a+b+c=1\}$  
say $\{(1,1,-1), (1,3,-2),\ldots\},$ etc. Denote $s_i=(x_i,y_i,z_i)$ $(i=1,2,3)$. Then we test all obtained solutions to find those distinct linear combinations (remember $a+b+c=1$) 

\[
as_1+bs_2+cs_3=(ax_1+bx_2+cx_3, ay_1+by_2+cy_3, az_1+bz_2+cz_3)
\]
such that $ax_1+bx_2+cx_3, ay_1+by_2+cy_3, az_1+bz_2+cz_3>0$ because any solution of the form $as_1+bs_2+cs_3$ with one zero component has been obtained already. 

As the simplest linear combination, we choose $(a,b,c)=(1,-1,1)$ and calculate 

\[
s_1-s_2+s_3=(x_1-x_2+x_3, y_1-y_2+y_3, z_1-z_2+z_3)
\]
such that $x_1-x_2+x_3, y_1-y_2+y_3, z_1-z_2+z_3>0$.

We now present the following conjecture for further discussion of finding all of nonnegative solutions of Diophantine equation \eqref{0.0}.

\medbreak

{\bf Conjecture:} For any solution $(\hat x,\hat y,\hat z)$ of $px+qy+lz=n$ with $\gcd(p,q,l)=1,$ there exist solutions $s_i,$ $i=1,2,3,$ of either $qy+lz=n,$ $px+lz=n,$ or $px+qy=n$ such that 

\be\label{conjecture}
(\hat x,\hat y,\hat z)=s_1-s_2+s_3,
\ee
where $s_i\in \widehat S_i,$ $i=1,2,3,$ in the last expression do not need to be different. Hence, if all three $s_i,$ $i=1,2,3,$ are the same, \eqref{conjecture} presents itself.
\medbreak

We will illustrate that the conjecture is reasonable by using some examples. Meanwhile, we will give an algorithm to apply \eqref{conjecture} to construct all of the solutions of the Diophantine equations in the examples.

We say the set $\widehat S_j$ is the smallest set of the collection $\{ \widehat S_i:i=1,2,3\},$ if the cardinal number $|\widehat S_j|=\min\{|\widehat S_1|, |\widehat S_2|, |\widehat S_3|\}$. We say the set $\widehat S_k$ is the next smallest set of the collection $\{ \widehat S_i:i=1,2,3\},$ if the cardinal number $|\widehat S_k|=\min\{|\widehat S_1|, |\widehat S_2|, |\widehat S_3|\}\backslash \{|\widehat S_j|\}$. We say $(x_{i},y_{i},z_{i})$ is the smallest element of a set $\widehat S_j=\{(x_{j},y_{j},z_{j}):j=1,2,\ldots\},$ if $x_{i}+y_{i}+z_{i}$ is the smallest possible number in the set $S_j$. Assume $p< q< \ell$ in \eqref{0.0}. Then, in general, $\widehat S_1$ is the smallest set and $\widehat S_2$ is the next smallest set. One can also check that this is the case, in equations \eqref{S3}, \eqref{S2}, and \eqref{S3}.  Our algorithm can be described based on Proposition \ref{pro:3.1} as follows:

\begin{enumerate}
\item [Step 1] Determine the smallest element of the smallest set.
\item [Step 2] Use the smallest element $(s_{11},s_{12},s_{13})$ determined in the first step to subtract all of the elements in the next smallest set, $\widehat S_2,$ provided that the differences of the third components are positive.
\item [Step 3] Add the resulting elements obtained in the second step to all of the elements in $\widehat S_3$, provided the resulting sums for all components are positive. 
\item [Step 4] The union of the sets $\widehat S_i,$ $i=1,2,3$, and those obtained in the third step consist of the whole solution set of \eqref{0.0}.
\end{enumerate}

As an example, we now use this algorithm to find the remaining part of the solution set of $x+2y+3z=14,$ besides those $14$ solutions shown in \eqref{S4}. Based on our definition, $(0,1,4)$ is the smallest element of the smallest set $\widehat S_3$ in \eqref{S3}. We use $(0,1,4)$ to subtract all of those elements in the next smallest set, $\widehat S_2$, in \eqref{S2}, such that the third components are positive, namely, 

\begin{align*}
(0,1,4)-(5,0,3)&=(-5,1,1), (0,1,4)-(8,0,2)=(-8,1,2),\\
(0,1,4)-(11,0,1)&=(-11,1,3), (0,1,4)-(14,0,0)=(-14,1,4).
\end{align*}
Following the algorithm, we add $(-5,1,1)$ to $(6,4,0), (8,3,0), (10,2,0),$ $(12,1,0),$ and $(14,0,0)$ in \eqref{S1}  to obtain the following $5$ solutions of $x+2y+3z=14$:

\[
(1,5,1), (3,4,1), (5,3,1), (7,2,1),\,\, \mbox{and} \,\, (9,1,1).
\]
By adding $(-8,1,2)$ to $(10,2,0), (12,1,0),$ and $(14,0,0)$ in \eqref{S1}, we obtain 
the following $3$ solutions of $x+2y+3z=14$:

\[
(2,3,2), (4,2,2),\,\, \mbox{and}\,\, (6,1,2).
\]
Adding $(-11,1,3)$ to $(12,1,0)$ and $(14,0,0)$ in \eqref{S1}, we get the $2$ more solutions of $x+2y+3z=14$ as 

\[
(1,2,3)\,\,\mbox{and}\,\, (3,1,3).
\]
Adding $(-14,1,4)$ to any element in \eqref{S1}, we can not get an element with positive components, so we can not get any more solutions. Thus, we obtain all $24$ solutions of $x+2y+3z=14$ by combining the above $10$ solutions and the $14$ solutions shown in \eqref{S4}. 

We use the exhaustive method to check our results and find 

\begin{align*}
S(0)&=\{ (x,y,0):x+2y=14\}=\{(0,7,0), (2,6,0), (4,5,0), (6,4,0),\\
& \quad\,\, (8,3,0), (10,2,0),  (12,1,0), (14,0,0)\},\\
S(3)&=\{(x,y,1):x+2y=14-3=11\}\\
&=\{ (11,0,1), (9,1,1),(7,2,1),(5,3,1),(3,4,1),(1,5,1)\}\\
S(6)&=\{(x,y,2):x+2y=14-6=8\}\\
&=\{(8,0,2), (6,1,2), (4,2,2), (2,3,2), (0,4,2)\}\\
S(9)&=\{(x,y,3):x+2y=14-9=5\}=\{(5,0,3), (3,1,3),(1,2,3)\}\\
S(12)&=\{(x,y,4):x+2y=14-12=2\}=\{(2,0,4),(0,1,4)\},
\end{align*}
where 

\begin{align*}
|S(0)|&=8=p_{1,2}(14),\quad  |S(3)|=6=p_{1,2}(11), \quad |S(6)|=5=p_{1,2}(8), \\
|S(9)|&=3=p_{1,2}(5),\quad |S(12)|=2=p_{1,2}(2).
\end{align*}

Our algorithm is efficient for solving many Diophantine equations. However, the algorithm may not work for some Diophantine equations \eqref{0.0} although we think the expression in \eqref{conjecture} can still be applied to find all of the solutions of \eqref{0.0}, with some modification of the above algorithm. Here is an example. Consider the Diophantine equation 

\be\label{finalex}
5x+7y+11z=71.
\ee

For reducing equation $5x+7y=71,$ from $5a_1\equiv 71 \equiv \modd{1} {7}$ and $7b_1\equiv 71 \equiv \modd{1} {5}$,   we obtain $a_1=3$ and $b_1=3$. Thus, 
\[
M=\frac{1}{35}(71-7\cdot 3-5\cdot 3)=1,
\]
which implies $5x+7y=71$ has $1+M=2$ nonnegative solutions $\widehat S_3=\{(7i+3,5(M-i)+3,0):i=0,1\}=\{ (3,8,0), (10,3,0)\}$. Similarly, for equation $7y+11z=71,$ the corresponding $a_1=7,$ $b_1=2,$ and 
\[
M=\frac{1}{77}(71-7\cdot 7-11\cdot 2)=0,
\]
which implies $7y+11z=71$ has $1+M=1$ nonnegative solution $\widehat S_1=\{(0, 11i+7, 7(M-i)+2):i=0\}=\{(0,7,2)\}$. For equation $5x+11z=71,$ the corresponding $a_1=1,$ $b_1=1,$ and 
\[
M=\frac{1}{55}(71-5\cdot 1-11\cdot 1)=1,
\]
which implies $5x+11z=71$ has $1+M=2$ solutions $\widehat S_2=\{(11i+1,0, 5(M-i)+1)\}=\{(1,0,6),(12,0,1)\}$. Denote 

\[
\widehat S=\cup^3_{i=1}\widehat S_i=\{(0,7,2), (1,0,6), (12, 0,1), (3,8,0), (10,3,0)\}.
\]
Then we pick up some three different elements from the five elements of the set $\widehat S$ to find other solutions of \eqref{finalex} by using \eqref{conjecture}. For instance, 

\begin{align*}
(1,0,6)-(0,7,2)+(3,8,0)&=(4,1,4),\\
(10,3,0)-(3,8,0)+(0,7,2)&=(7,2,2),\\
(12,0,1)-(10,3,0)+(0,7,2)&=(2,4,3),\\
(12,0,1)-(10,3,0)+(3,8,0)&=(5,5,1).
\end{align*}
By using Theorem \ref{thm:1.2}, we may find equations $5x+7y=71-11z$ have $2,$ $2,$ $2,$ $1,$ $1,$ $0,$ and $1$ nonnegative solutions for $z=0,1,2,3,4,5,$ and $6,$ respectively. Hence, $5x+7y+11z=71$ has $9$ nonnegative solutions, and all of them can be presented by using \eqref{conjecture}. It worth mentioning that for $z=5,$ $5x+7y=71-11\cdot 5=16$ has no nonnegative solution from Theorem \ref{thm:1.2}, since $M(5)=(16-5a_1(5)-7b_1(5))/35=(16-5\cdot 6-7\cdot 3)/35<0$. 

Here, we have two comments. First, some three elements may not give a nonnegative solution of \eqref{finalex} by using 
\eqref{conjecture}, say $(12,0,1)-(0,7,2)+(3,8,0)=(13,1,-1)$. Secondly, it is clear that each element in $\widehat S$ can also be presented by using \eqref{conjecture}. Hence we have the following 
consequence of the conjecture:
\medbreak

{\bf Consequence of the conjecture:} Denote $|\widehat S_i|=N_i$ and $\hat N=N_1+N_2+N_3$. Then the number of solutions of \eqref{0.0} has bounds

\[
0\leq N(p,q,\ell,n)\leq 3\binom{\hat N}{3},
\]
which is implied by the conjecture if it is true.

\end{document}